\documentclass[reqno,12pt]{amsart}

\usepackage{amsmath, epsfig, cite}
\usepackage{amssymb}
\usepackage{amsfonts}
\usepackage{latexsym}
\usepackage{amsthm}
\newtheorem{theorem}{Theorem}[section]

\newtheorem{corollary}[theorem]{Corollary}
\newtheorem{lemma}[theorem]{Lemma}

\theoremstyle{remark}

\numberwithin{equation}{section}

\allowdisplaybreaks

\author{Xiaoxia Wang}
\address{DEPARTMENT OF MATHEMATICS, SHANGHAI UNIVERSITY, SHANGHAI 200444, P. R. CHINA}
\email{$^*$ Corresponding author. xiaoxiawang@shu.edu.cn (X. Wang), xchangi@shu.edu.cn (C. Xu).}

\author{Chang Xu$^*$}
\title{A $q$-supercongruence modulo the fourth power of a cyclotomic polynomial }
\thanks{This work is supported by Natural Science Foundations of Shanghai (22ZR1424100).}


\begin{document}

\begin{abstract}
In this paper, a new $q$-supercongruence with two free parameters
modulo the fourth power of a cyclotomic polynomial is obtained. Our main auxiliary tools are  Watson's $_8\phi_7$ transformation formula for basic hypergeometric series, the `creative microscoping' method recently introduced by Guo and Zudilin and the Chinese remainder theorem for coprime polynomials.
By taking suitable parameter substitutions in the established $q$-supercongruence, some nice congruences involving the Bernoulli numbers are derived.
\end{abstract}
\maketitle

\section{Introduction}
In 1997, Van Hamme \cite{Van} conjectured $13$  Ramanujan-type supercongruences which were labeled as (A.2)--(M.2).
Van Hamme's (D.2) states that for primes $p\equiv1\pmod 6$,
\begin{equation*}
\sum_{k=0}^{(p-1)/3}(6k+1)\frac{(1/3)^6_k}{k!^6}
\equiv-p\Gamma_p\left(1/3\right)^9 \pmod{p^4}.
\end{equation*}
Here and in what follows, $p$ is a prime, $(x)_0=1$, $(x)_n=x(x+1)\cdots(x+n-1)$ stands for the Pochhammer's symbol and $\Gamma_p(x)$ is the $p$-adic Gamma function.
In 2006, making use of Dougall's formula, Long and Ramakrishna \cite{Long1} gave an extension of Van Hamme's (D.2):
\begin{equation*}\label{E.3}
\sum_{k=0}^{p-1}(6k+1)\frac{(1/3)^6_k}{k!^6}
\equiv \left\{
\begin{aligned}
-p\Gamma_p(1/3)^9\pmod{p^6}, \qquad\mathrm{if}\quad p\equiv1\pmod6,\\
-\frac{10}{27}p^4\Gamma_p(1/3)^9\pmod {p^6}, \qquad \mathrm{if} \quad p\equiv5\pmod6.
\end{aligned}
\right.
\end{equation*}
Similarly, Liu \cite{Liu3} established a new supercongruence: for $p\geq5$,
\begin{equation}\label{E.4}
\sum_{k=0}^{p-1}(6k-1)\frac{(-1/3)^6_k}{k!^6}
\equiv \left\{
\begin{aligned}
140p^4\Gamma_p(2/3)^9\pmod {p^5}, \qquad \mathrm{if} \quad p\equiv1\pmod6,\\
378p\Gamma_p(2/3)^9\pmod{p^5}, \qquad\mathrm{if}\quad p\equiv5\pmod6.
\end{aligned}
\right.
\end{equation}
Also, Guo and Schlosser \cite{GAS2} proposed one conjecture as follows:
for  $p\equiv 2\pmod3$,
\begin{equation}\label{EQU1}
\sum_{k=0}^{(p+1)/3}(6k-1)\frac{(-1/3)^4_k(1)_{2k}}{(1)^4_k(-2/3)_{2k}}\equiv p\pmod{p^3}.
\end{equation}
By using the hypergeometric series identities and $p$-adic Gamma functions, Jana and Kalita \cite{JanaAndKalita} first confrimed the  supercongruence \eqref{EQU1}.
Later, based on combinatorial identities arising from symbolic summation,  Liu \cite{LJC2} provided a stronger version of \eqref{EQU1}:
for odd primes $p\equiv2\pmod3$,
 \begin{equation}\label{EQU2}
\sum_{k=0}^{(p+1)/3}(6k-1)\frac{(-1/3)^4_k(1)_{2k}}{(1)^4_k(-2/3)_{2k}}\equiv p-p^3\left(\frac{1}{9}B_{p-2}\left(1/3\right)-2\right)\pmod{p^4},
\end{equation}
where the {\em Bernoulli polynomials} are given by
\begin{equation*}
\frac{te^{xt}}{e^t-1}=\sum_{k=0}^{\infty}B_k(x)\frac{t^k}{k!}.
\end{equation*}

During the past few years, there has been an increasing attention to the issue of finding $q$-analogues of
 congruences and supercongruences.
The reader may be referred to  \cite{LW,Guo,Tauraso,WY2,ym,Wei1,Wei2,Wei3,Liu2,Xu_Wang} for some of their work.
Recently, in  \cite{GAS2}, Guo and Schlosser gave a partial $q$-analogue of supercongruence \eqref{EQU1}:
for  integers $n>2$ with $n\equiv 2\pmod3$,
\begin{equation}\label{EQU3}
\sum_{k=0}^{(n+1)/3}[6k-1]\frac{\left(q^{-1};q^3\right)^4_k\left(q^3;q^3\right)_{2k}}{\left(q^3;q^3\right)^4_{k}
\left(q^{-2};q^3\right)_{2k}}q^{4k}\equiv0\pmod{\Phi_n(q)}.
\end{equation}
In the same paper,  they also conjectured a partial $q$-analogue of Liu's supercongruence \eqref{E.4} as follows:
for  integers $n>1$ with $n\equiv1\pmod3$,
\begin{equation*}\label{EQU4}
\sum_{k=0}^{n-1}[6k-1]\frac{\left(q^{-1};q^3\right)^6_k}{\left(q^3;q^3\right)^6_{k}}q^{9k}\equiv0\pmod{[n]\Phi_n(q)^3}.
\end{equation*}
Here and in what follows, the {\em $q$-shifted factorial} is defined as
$(a;q)_0=1$ and $(a;q)_n=\left(1-a\right)\left(1-aq\right)\cdots\left(1-aq^{n-1}\right)$ with $n\in \mathbb{Z}^{+}$.
For brevity, its product form can be written as
$(a_1,a_2,\ldots,a_m;q)_n=(a_1;q)_n (a_2;q)_n\cdots (a_m;q)_n$. And $[n]=[n]_q=1+q+\cdots+q^{n-1}$ denotes the {\em $q$-integer}.
Moreover, $\Phi_n(q)$ represents the $n$-th {\em cyclotomic polynomial} in $q$.

Motivated by the work just mentioned, in this paper, we shall establish a new $q$-supercongruence with two free parameters,
from which we can deduce a partial $q$-analogue of Liu's congruence \eqref{E.4} and a $q$-analogue of Liu's congruence \eqref{EQU2}.

The rest of this paper is arranged as follows. Our main results will be shown in the next section.
Then the  proof of our $q$-supercongruence will be presented in Section \ref{se2},
where the `creative microscoping' method introduced by Guo and Zudilin \cite{GuoZu} and the Chinese remainder theorem for coprime
 polynomials will be  used.
\section{Main results}
\begin{theorem}\label{T.1}
Let $n\equiv 2 \pmod 3$ be a positive integer. Then, modulo $[n]\Phi_n(q)^3$,
\begin{align*}
&\sum_{k=0}^M[6k-1]\frac{\left(q^{-1};q^3\right)^4_k\left(cq^{-1},eq^{-1};q^3\right)_k}
{{\left(q^3;q^3\right)^4_k}\left(q^3/c,q^3/e;q^3\right)_k}\left(\frac{q^9}{ce}\right)^k\nonumber\\
&\equiv[n]q^{(n+1)/3}\frac{\left(q^{-2};q^3\right)_{(n+1)/3}}{\left(q^3;q^3\right)_{(n+1)/3}}\left
(1-[n]^2\sum_{i=1}^{(n+1)/3}\frac{q^{3i}}{[3i]^2}\right)
\sum_{k=0}^{(n+1)/3}\frac{\left(q^4/{ce};q^3\right)_k\left(q^{-1};q^3\right)^3_k}{\left(q^3/c,q^3/e,q^3,q^{-2};q^3\right)_k}q^{3k},\label{E.1}
\end{align*}
here and in what follows $M=(n+1)/3$ or $n-1$.
\end{theorem}
Setting $c\to 1$, $e\to 1$ in Theorem \ref{T.1}, we obtain a  partial $q$-analogue of Liu's congruence \eqref{E.4} as follows.
\begin{corollary}
Let $n\equiv 2 \pmod 3$ be a positive integer. Then, modulo ${[n]\Phi_n(q)^3}$,
\begin{align}
&\sum_{k=0}^M[6k-1]\frac{\left(q^{-1};q^3\right)^6_k}{\left(q^3;q^3\right)^6_k}q^{9k}\nonumber\\
&\:\:\equiv[n]q^{(n+1)/3}\frac{\left(q^{-2};q^3\right)_{(n+1)/3}}{\left(q^3;q^3\right)_{(n+1)/3}}\left(1-[n]^2\sum_{i=1}^{(n+1)/3}\frac{q^{3i}}{[3i]^2}\right)
\sum_{k=0}^{(n+1)/3}\frac{\left(q^4;q^3\right)_k\left(q^{-1};q^3\right)^3_k}{\left(q^3;q^3\right)^3_k\left(q^{-2};q^3\right)_k}q^{3k}.
\end{align}
\end{corollary}
Moreover, putting $c=q^{4}$ and $e=q^{\frac{5}{2}}$ in Theorem 2.1, we have the following result.
\begin{corollary}
Let $n \equiv 2\pmod 3$ be a positive integer. Then, modulo $[n] \Phi_{n}(q)^{3}$,
\begin{align}
&\sum_{k=0}^{M}[6k-1]\frac{\left(q^{-1};q^3\right)^3_k\left(q^{3/2};q^3\right)_k}{\left(q^3;q^3\right)^3_k\left(q^{1/2};q^3\right)_k}q^{5k/2}\nonumber\\
&\:\:\equiv[n]q^{(n+1)/3}\frac{\left(q^{-2};q^3\right)_{(n+1)/3}}{\left(q^3;q^3\right)_{(n+1)/3}}\left(1-[n]^2\sum_{i=1}^{(n+1)/3}\frac{q^{3i}}{[3i]^2}\right)
\sum_{k=0}^{(n+1)/3}\frac{\left(q^{-5/2};q^3\right)_k\left(q^{-1};q^3\right)^2_k}{\left(q^3,q^{1/2},q^{-2};q^3\right)_k}q^{3k}. \label{Q1}
\end{align}
\end{corollary}
Recently, Liu \cite{LJC2} proved that, for $p\equiv2\pmod3$,
\begin{equation}\label{EQU9}
\sum_{k=0}^{(p+1)/3}\left(6k-1\right)\frac{\left(-1/3\right)^4_k(1)_{2k}}{(1)^4_k\left(-2/3\right)_{2k}}
=\sum_{k=0}^{(p+1)/3}\left(6k-1\right)\frac{\left(-1/3\right)^3_k(1/2)_{k}}{(1)^3_k\left(1/6\right)_{k}}.
\end{equation}
From \eqref{EQU9}, we can deduce that \eqref{Q1} is a $q$-analogue of Liu's congruence \eqref{EQU2}.

Furthermore, letting $q\rightarrow q^2$ and $c=e=q^7$ in Theorem \ref{T.1}, we obtain a new result as follows.
\begin{corollary}\label{Cor1}
Let  $n\equiv2\pmod3$ be a positive integer. Then, modulo ${[n]_{q^2}\Phi_n(q^2)^3}$,
\begin{equation*}\label{EQU6}
\sum_{k=0}^M[6k-1]_{q^2}[6k-1]^2\frac{\left(q^{-2};q^6\right)^4_k}{{\left(q^6;q^6\right)^4_k}}q^{4k}
\equiv\frac{-2[n]_{q^2}q^{\frac{2n-7}{3}}\left(q^{-4};q^6\right)_{(n+1)/3}}{\left(1+q^{-2}\right)\left(q^{6};q^6\right)_{(n+1)/3}}
\left(1-[n]^2_{q^2}\sum_{i=1}^{(n+1)/3}\frac{q^{6i}}{[3i]^2_{q^2}}\right).
\end{equation*}
\end{corollary}
By using the following  congruence from \cite{Lehmer}:
\begin{equation}
\sum_{k=1}^{\lfloor p/3\rfloor}\frac{1}{k^2}\equiv\frac{1}{2}\left(\frac{-3}{p}\right)B_{p-2}\left(1/3\right)\pmod{p},
\end{equation}
and letting $n=p$ with  $p\equiv2\pmod3$ and $q\rightarrow1$ in Corollary \ref{Cor1}, we can get the supercongruence:
for $p\equiv2\pmod3$, modulo $p^4$,
\begin{align}
\sum_{k=0}^{(p+1)/3}(6k-1)^3\frac{(-1/3)^4_k}{k!^4}
\equiv(-1)^{(p-2)/3}p\Gamma^2_p\left(2/3\right)\left(1-p^2-\frac{p^2}{18}\left(\frac{-3}{p}\right)B_{p-2}\left(1/3\right)\right),\label{EQU7}
\end{align}
where $\left(\frac{\cdot}{p}\right)$ denotes the Legendre symbol.

Moreover, taking $ce=q^4$ in Theorem \ref{T.1},
 we get the following  $q$-supercongruence.
 \begin{corollary}\label{Cor2}
Let $n\equiv 2 \pmod 3$ be a positive integer. Then, modulo ${[n]\Phi_n(q)^3}$,
\begin{align}
\sum_{k=0}^M[6k-1]\frac{\left(q^{-1};q^3\right)^4_k}{\left(q^3;q^3\right)^4_k}q^{5k}
\equiv[n]q^{(n+1)/3}\frac{\left(q^{-2};q^3\right)_{(n+1)/3}}{\left(q^3;q^3\right)_{(n+1)/3}}
\left(1-[n]^2\sum_{i=1}^{(n+1)/3}\frac{q^{3i}}{[3i]^2}\right).\nonumber
\end{align}
\end{corollary}
Letting $n=p$ with  $p\equiv2\pmod3$ and $q\rightarrow1$ in Corollary \ref{Cor2}, we obtain a new congruence: for  $p\equiv2\pmod3$,
 modulo $p^4$,
\begin{align}
\sum_{k=0}^{(p+1)/3}(6k-1)\frac{(-1/3)^4_k}{k!^4}
\equiv (-1)^{(p+1)/3}p\Gamma^2_p\left(2/3\right)\left(1-p^2-\frac{p^2}{18}\left(\frac{-3}{p}\right)B_{p-2}\left(1/3\right)\right).\label{EQU8}
\end{align}
Combining  \eqref{EQU7} and \eqref{EQU8}, we can get a new and rare supercongruence:  for  $p\equiv2\pmod3$,
\begin{align}
\sum_{k=0}^{(p+1)/3}(6k-1)(18k^2-6k+1)\frac{(-1/3)^4_k}{k!^4}\equiv0\pmod{p^4}.
\end{align}

\section{Proof of Theorem \ref{T.1} }\label{se2}
In fact, the proof of Theorem \ref{T.1} can be transformed into confirming the following generalized theorem.
\begin{theorem}\label{Thm1}
Let $n>1$, $d\ge 2$ be integers with  $n\equiv r \pmod d$ and $r\in \left\{1,-1\right\}$.
Then, modulo ${[n]\Phi_n(q)^3}$,
\begin{equation}\label{eq1}
\begin{split}
&\sum_{k=0}^W[2dk+r]\frac{\left(q^r;q^d\right)^4_k\left(cq^r,eq^r;q^d\right)_k}{{\left(q^d;q^d\right)^4_k}
\left(q^d/c,q^d/e;q^d\right)_k}(ce)^{-k}q^{(2d-3r)k}\\
&\equiv[n]\frac{q^{r(r-n)/d}}{\left(q^d;q^d\right)_{(n-r)/d}}\left
(1-[n]^2\sum_{i=1}^{(n-r)/d}\frac{q^{di}}{[di]^2}\right)\\
&\quad\quad\times\sum_{k=0}^{(n-r)/d}\frac{\left(q^{2r+dk};q^d\right)_{(n-r)/d-k}(q^{d-r}/{ce};q^d)_k\left(q^{r};q^d\right)^3_k}
{\left(q^d/c,q^d/e,q^d;q^d\right)_k}q^{dk},
\end{split}
\end{equation}
here and in what follows $W=(n-r)/d$ or $n-1$.
\end{theorem}
Clearly, when $d=3$, $r=-1$, Theorem \ref{Thm1} reduces to Theorem \ref{T.1}. Actually, by making appropriate parameter substitutions in Theorem \ref{Thm1},
more results can be obtained. For example, letting $d=3, r=1$, $c\rightarrow1, e\rightarrow1$ and $q\rightarrow1$ in Theorem \ref{Thm1}, we can reprove Van Hamme's (D.2).

In the process of proving Theorem \ref{Thm1}, we shall utilize Watson's $_8\phi_7$ transformation formula \cite{GR}:
\begin{align}\label{eq6}
& _{8}\phi_{7}\!\left[\begin{array}{cccccccc}
a,& qa^{\frac{1}{2}},& -qa^{\frac{1}{2}}, & b,    & c,    & d,    & e,    & q^{-n} \\
  & a^{\frac{1}{2}}, & -a^{\frac{1}{2}},  & aq/b, & aq/c, & aq/d, & aq/e, & aq^{n+1}
\end{array};q,\, \frac{a^2q^{n+2}}{bcde}
\right] \notag\\[5pt]
&\quad =\frac{(aq, aq/de;q)_n}
{(aq/d, aq/e;q)_n}
\,{}_{4}\phi_{3}\!\left[\begin{array}{c}
aq/bc,\ d,\ e,\ q^{-n} \\
aq/b,\, aq/c,\, deq^{-n}/a
\end{array};q,\, q
\right].
\end{align}
Here, the basic hypergeometric series $_{r+1}\phi_r$, following Gasper and Rahman\cite{GR}, is defined as
	$$
	_{r+1}\phi_{r}\left[\begin{array}{c}
		a_1,a_2,\ldots,a_{r+1}\\
		b_1,b_2,\ldots,b_{r}
	\end{array};q,\, z
	\right]
	=\sum_{k=0}^{\infty}\frac{(a_1,a_2,\ldots, a_{r+1};q)_k z^k}
	{(q,b_1,\ldots,b_{r};q)_k}, \qquad\mathrm{for}\quad 0<|q|<1.
	$$

Before proving Theorem \ref{Thm1}, we first list the following two related results, which have been proved in \cite{L}.
\begin{lemma}\label{la3}
Let $d$, $n$ be positive integers with $gcd(d,n)=1$. Let $r$ be an integer and $a$, $b$, $c$, $e$ be indeterminates. Then, modulo $[n]$,
\begin{align*}
\sum_{k=0}^{m_1}[2dk+r]\frac{\left(q^r,cq^r,eq^r,q^r/b,aq^r,q^r/a;q^d\right)_k}
{\left(q^d,q^d/c,q^d/e,bq^d,q^d/a,aq^d;q^d\right)_k}
\left(\frac{b}{ce}\right)^kq^{(2d-3r)k}
&\equiv0,\\
\sum_{k=0}^{n-1}[2dk+r]\frac{\left(q^r,cq^r,eq^r,q^r/b,aq^r,q^r/a;q^d\right)_k}
{\left(q^d,q^d/c,q^d/e,bq^d,q^d/a,aq^d;q^d\right)_k}
\left(\frac{b}{ce}\right)^kq^{(2d-3r)k}
&\equiv0,
\end{align*}
where $0\le m_1\le n-1$ and $dm_1\equiv -r\pmod {n}$.
\end{lemma}
\begin{lemma}\label{la1}
Let $n>1$, $d\ge 2$, $r$ be integers with $gcd(r,d)=1$ and $n\equiv r\pmod d$ such that $n+d-nd\le r\le n$.
 Then, modulo $\Phi_n(q)(1-aq^n)(a-q^n)$,
\begin{align}
&\sum_{k=0}^{(n-r)/d}[2dk+r]\frac{(q^r,cq^r,eq^r,q^r/b,aq^r,q^r/a;q^d)_k}
{(q^d,q^d/c,q^d/e,bq^d,q^d/a,aq^d;q^d)_k}
\left(\frac{b}{ce}\right)^kq^{(2d-3r)k}\nonumber\\
&\equiv[n]\left(\frac{b}{q^r}\right)^{(n-r)/d}
\frac{\left(q^{2r}/b;q^d\right)_{(n-r)/d}}{\left(bq^d;q^d\right)_{(n-r)/d}}
\sum_{k=0}^{(n-r)/d}\frac{\left(q^{d-r}/ce,q^r/b,aq^r,q^r/a;q^d\right)_k}{\left(q^d,q^d/c,q^d/e,q^{2r}/b;q^d\right)_k}q^{dk}.\label{eq2}
\end{align}
\end{lemma}

In order to complete our proof of Theorem \ref{Thm1}, we still need the following lemma.
\begin{lemma}\label{la2}
Let $n>1$, $d\ge2$ be  integers with $n\equiv r\pmod d$ and $r\in \left\{1, -1\right\}$. Then, modulo $b-q^n$,
\begin{align}
&\sum_{k=0}^{W}[2dk+r]\frac{\left(q^r,cq^r,eq^r,q^r/b,aq^r,q^r/a;q^d\right)_k}
{\left(q^d,q^d/c,q^d/e,bq^d,q^d/a,aq^d;q^d\right)_k}
\left(\frac{b}{ce}\right)^kq^{(2d-3r)k}\nonumber\\
&\equiv[n]\frac{\left(q^r,q^{d-r};q^d\right)_{(n-r)/d}}{\left(q^d/a,aq^d;q^d\right)_{(n-r)/d}}
\sum_{k=0}^{(n-r)/d}\frac{\left(q^{d-r}/{ce},aq^r,q^r/a,q^r/b;q^d\right)_k}{\left(q^d,q^d/c,q^d/e,q^{2r}/b;q^d\right)_k}q^{dk},\label{eq3}
\end{align}
where $W=(n-r)/d$ or $n-1$.
\end{lemma}
\begin{proof}
Letting $q\to q^d$, $n\to (n-r)/d$, $a=q^r$, $b=cq^r$, $c=eq^r$, $d=aq^r$ and $e=q^r/a$ in Watson's $_8\phi_7$ transformation formula \eqref{eq6}, we have
\begin{align*}
&\sum_{k=0}^{(n-r)/d}[2dk+r]\frac{(q^r,cq^r,eq^r,q^{r-n},aq^r,q^r/a;q^d)_k}
{(q^d,q^d/c,q^d/e,q^{d+n},q^d/a,aq^d;q^d)_k}
\left(\frac{q^{2d+n-3r}}{ce}\right)^k\\
&\quad=[n]\frac{(q^r,q^{d-r};q^d)_{(n-r)/d}}{(q^d/a,aq^d;q^d)_{(n-r)/d}}
\sum_{k=0}^{(n-r)/d}\frac{(q^{d-r}/{ce},aq^r,q^r/a,q^{r-n};q^d)_k}{(q^d,q^d/c,q^d/e,q^{2r-n};q^d)_k}q^{dk}.
\end{align*}
 In light of the fact that $(q^{r-n};q^d)_k=0$ for $n-1\ge k>(n-r)/d$, we confirm the correction of \eqref{eq3}.
\end{proof}
Now, we present a parametric generalization of Theorem \ref{Thm1}.
\begin{theorem}\label{Thm2}
Let $n>1$, $d\ge2$ be  integers with $n\equiv r\pmod d$ and $r\in\left\{1, -1\right\}$. Then, modulo $\Phi_n(q)^2(1-aq^n)(a-q^n)$,
\begin{align}\label{eq5}
&\sum_{k=0}^{(n-r)/d}[2dk+r]\frac{\left(q^r;q^d\right)^2_k\left(cq^r,eq^r,aq^r,q^r/a;q^d\right)_k}
{\left(q^d;q^d\right)^2_k\left(q^d/c,q^d/e,q^d/a,aq^d;q^d\right)_k}
\left(\frac{q^{2d-3r}}{ce}\right)^k\nonumber\\
&\quad\equiv[n]Q_q(a,n)
\sum_{k=0}^{(n-r)/d}\frac{\left(q^{2r+dk};q^d\right)_{(n-r)/d-k}\left(q^{d-r}/{ce},aq^r,q^r/a,q^r;q^d\right)_k}{\left(q^d,q^d/c,q^d/e;q^d\right)_k}q^{dk},
\end{align}
where
\begin{align*}
Q_q(a,n)&=\frac{{q^{r(r-n)/d}}\left(1-aq^n\right)\left(a-q^n\right)}{(1-a)^2}
\left\{\frac{1}{\left(q^d;q^d\right)_{(n-r)/d}}-\frac{\left(q^d;q^d\right)_{(n-r)/d}}{\left(q^d/a,aq^d;q^d\right)_{(n-r)/d}}\right\}\\
&\quad+\frac{q^{r(r-n)/d}}{\left(q^d;q^d\right)_{(n-r)/d}}.
\end{align*}
\end{theorem}
\begin{proof}
It is easy to see that $\Phi_n(q)\left(1-aq^n\right)\left(a-q^n\right)$ and $b-q^n$ are relatively prime polynomials. Noting the relations
\begin{align*}
\frac{\left(b-q^n\right)\left(ab-1-a^2+aq^n\right)}{(a-b)(1-ab)}&\equiv 1 \pmod{(1-aq^n)(a-q^n)},\\
\frac{\left(1-aq^n\right)\left(a-q^n\right)}{(a-b)(1-ab)}&\equiv 1 \pmod{b-q^n},
\end{align*}
and employing the Chinese remainder theorem for coprime polynomials, we arrive at the following result
from Lemma \ref{la1} and Lemma \ref{la2}: modulo $\Phi_n(q)(1-aq^n)(a-q^n)(b-q^n)$,
\begin{align}
&\sum_{k=0}^{(n-r)/d}[2dk+r]\frac{\left(q^r,cq^r,eq^r,q^r/b,aq^r,q^r/a;q^d\right)_k}
{\left(q^d,q^d/c,q^d/e,bq^d,q^d/a,aq^d;q^d\right)_k}
\left(\frac{b}{ce}\right)^kq^{(2d-3r)k}\nonumber\\
&\quad\equiv[n]\theta_q(a,b,n)\sum_{k=0}^{(n-r)/d}\frac{\left(q^{d-r}/{ce},aq^r,q^r/a,q^r/b;q^d\right)_k}
{\left(q^d,q^d/c,q^d/e,q^{2r}/b;q^d\right)_k}q^{dk},\label{eq4}
\end{align}
where the notation $\theta_q(a,b,n)$ on the right-hand side denotes
\begin{align*}
\theta_q(a,b,n)&=\frac{(b-q^n)\left(ab-1-a^2+aq^n\right)}{(a-b)(1-ab)}\frac{(b/{q^r})^{(n-r)/d}\left(q^{2r}/b;q^d\right)_{(n-r)/d}}
{\left(bq^d;q^d\right)_{(n-r)/d}}\\
&\quad+\frac{(1-aq^n)(a-q^n)}{(a-b)(1-ab)}\frac{\left(q^r,q^{d-r};q^d\right)_{(n-r)/d}}{\left(q^d/a,aq^d;q^d\right)_{(n-r)/d}}.
\end{align*}
It is not difficult to see that
\begin{align*}
(q^{d-r};q^d)_{(n-r)/d}&=\left(1-q^{d-r}\right)\left(1-q^{2d-r}\right)\cdots\left(1-q^{n-2r}\right)\\
&\equiv(1-bq^{d-r-n})(1-bq^{2d-r-n})\cdots(1-bq^{-2r})\\
&\equiv(-1)^{(n-r)/d}b^{(n-r)/d}q^\frac{(n-r)(d-n-3r)}{2d}(q^{2r}/b;q^d)_{(n-r)/d}\pmod {b-q^n},\\
(q^{r};q^d)_{(n-r)/d}&=(1-q^r)(1-q^{d+r})\cdots(1-q^{n-d})\\
&\equiv(1-bq^{r-n})\left(1-bq^{d+r-n}\right)\cdots\left(1-bq^{-d}\right)\\
&\equiv(-1)^{(n-r)/d}q^\frac{(n-r)(n-d+r)}{2d}(q^d/b;q^d)_{(n-r)/d}\pmod{b-q^n}.
\end{align*}
Therefore, we can rewrite \eqref{eq4} as, modulo $\Phi_n(q)(1-aq^n)(a-q^n)(b-q^n)$,
\begin{align}
&\sum_{k=0}^{(n-r)/d}[2dk+r]\frac{\left(q^r,cq^r,eq^r,q^r/b,aq^r,q^r/a;q^d\right)_k}
{\left(q^d,q^d/c,q^d/e,bq^d,q^d/a,aq^d;q^d\right)_k}
\left(\frac{b}{ce}\right)^kq^{(2d-3r)k}\nonumber\\
&\quad \equiv[n]\Omega_q(a,b,n)\sum_{k=0}^{(n-r)/d}\frac{\left(q^{2r+dk}/b;q^d\right)_{(n-r)/d-k}
\left(q^{d-r}/{ce},aq^r,q^r/a,q^r/b;q^d\right)_k}{\left(q^d,q^d/c,q^d/e;q^d\right)_k}q^{dk},\label{eq9}
\end{align}
where the notation $\Omega_q(a,b,n)$ on the right-hand side denotes
\begin{align*}
\Omega_q(a,b,n)&=\frac{\left(b-q^n\right)\left(ab-1-a^2+aq^n\right)
\left(b/{q^r}\right)^{(n-r)/d}}{(a-b)(1-ab)\left(bq^d;q^d\right)_{(n-r)/d}}\\
&\quad +\frac{\left(1-aq^n\right)\left(a-q^n\right)\left(b/{q^r}\right)^{(n-r)/d}
\left(q^d/b;q^d\right)_{(n-r)/d}}{(a-b)(1-ab)\left(q^d/a,aq^d;q^d\right)_{(n-r)/d}}.
\end{align*}
It is easy to say that the limit of $b-q^n$ as $b\to 1$ has the factor $\Phi_n(q)$. Meanwhile, since $n\equiv r\pmod d$, i.e., $gcd(d,n)=1$,
the factor $(bq^d;q^d)_{(n-r)/d}$
in the denominator of the left-hand side of \eqref{eq9} as $b\to 1$ is relatively prime to $\Phi_n(q)$.
Thus, letting $b\to 1$ in \eqref{eq9}, we conclude that \eqref{eq5} is true modulo $\Phi_n(q)^2(1-aq^n)(a-q^n)$ with the relation:
\begin{equation*}
(1-q^n)(1+a^2-a-aq^n)=(1-a)^2+(1-aq^n)(a-q^n).
\end{equation*}
\end{proof}
\begin{proof}[Proof of Theorem \ref{Thm1}]
By the L'Hospital rule, we have
\begin{align*}
&\lim_{a\to 1}\frac{(1-aq^n)(a-q^n)}{(1-a)^2}\left\{\frac{1}{\left(q^d;q^d\right)_{(n-r)/d}}-\frac{\left(q^d;q^d\right)_{(n-r)/d}}
{\left(aq^d,q^d/a;q^d\right)_{(n-r)/d}}\right\}
\\
&=-\frac{[n]^2}{\left(q^d;q^d\right)_{(n-r)/d}}\sum_{i=1}^{(n-r)/d}\frac{q^{di}}{[di]^2}.
\end{align*}
Letting $a\to 1$ in Theorem \ref{Thm2} and utilizing the above limit, we deduce that \eqref{eq1} is true
modulo $\Phi_n(q)^4$ by noticing that $\left(q^{r}; q^d\right)^4_k \equiv 0 \pmod {\Phi_n(q)^4}$ for
$(n-r)/d < k \le n-1$.
From Lemma \ref{la3} with $r\in\left\{1, -1\right\}$ and $a=b=1$, we conclude that the congruence \eqref{eq1}
 holds modulo $[n]$. Since the least common multiple of $[n]$ and $\Phi_n(q)^4$ is $[n]\Phi_n(q)^3$,
 we obtain the desired result.
\end{proof}




\end{document}